\documentclass[12pt,a4paper,psamsfonts]{amsart}
\usepackage{amssymb,amscd,amsxtra,calc}
\usepackage{cmmib57}
\usepackage{multirow}
\usepackage[all]{xy}
\usepackage[colorlinks,linkcolor=blue,anchorcolor=blue,citecolor=green]{hyperref}
\theoremstyle{plain}
    \newtheorem{thm}{Theorem}[section]

    \newtheorem{corollary}[thm]{Corollary}
    
    \newtheorem{lemma}[thm]{Lemma}
    
    \newtheorem{question}[thm]{Question}
    \newtheorem{theorem}[thm]{Theorem}

\theoremstyle{definition}
    \newtheorem{definition}[thm]{Definition}

    \newtheorem{remark}[thm]{Remark}
\theoremstyle{remark}

    \newtheorem{setup}[thm]{}

\newcommand{\C}{\mathbb{C}}

\newcommand{\PP}{\mathbb{P}}

\newcommand{\Q}{\mathbb{Q}}

\newcommand{\Z}{\mathbb{Z}}

\newcommand{\SG}{\mathcal{G}}

\newcommand{\aff}{\operatorname{aff}}

\newcommand{\ant}{\operatorname{ant}}
\newcommand{\Aut}{\operatorname{Aut}}
\newcommand{\Bir}{\operatorname{Bir}}

\newcommand{\GL}{\operatorname{GL}}
\newcommand{\Fd}{\operatorname{F_d}}

\newcommand{\Ker}{\operatorname{Ker}}

\newcommand{\NS}{\operatorname{NS}}

\newcommand{\Bd}{\operatorname{Bd}}
\newcommand{\Rk}{\operatorname{Rk_f}}
\newcommand{\Rank}{\operatorname{Rank}}

\newcommand{\Con}{\operatorname{C}}
\newcommand{\J}{\operatorname{J}}
\newcommand{\N}{\operatorname{N}}
\newcommand{\Se}{\operatorname{S}}

\newcommand{\Hilb}{\operatorname{Hilb}}

\begin{document}

\title[Jordan property]
{Jordan property for non-linear algebraic groups and projective varieties}

\author{Sheng Meng, De-Qi Zhang}

\address
{
\textsc{Department of Mathematics} \endgraf
\textsc{National University of Singapore, 10 Lower Kent Ridge Road,
Singapore 119076
}}
\email{ms@u.nus.edu}
\address
{
\textsc{Department of Mathematics} \endgraf
\textsc{National University of Singapore, 10 Lower Kent Ridge Road,
Singapore 119076
}}
\email{matzdq@nus.edu.sg}

\begin{abstract}
A century ago, Camille Jordan proved that the complex general linear group $\GL_n(\C)$ has the Jordan property:
there is a Jordan constant $\Con_n$ such that every finite subgroup $H \le \GL_n(\C)$
has an abelian subgroup $H_1$ of index $[H : H_1] \le \Con_n$.
We show that every connected algebraic group $G$ (which is not necessarily linear) has the Jordan property with
the Jordan constant depending only on $\dim G$, and that the full automorphism group $\Aut(X)$ of every
projective variety $X$ has the Jordan property.
\end{abstract}




\maketitle

\section{Introduction}

We work over an algebraically closed field $k$ of characteristic zero unless explicitly stated otherwise.
A group $G$ is a {\it Jordan group} if there is a constant $\J$, called a {\it Jordan constant},
such that $G$ satisfies the following {\it Jordan property}:
every finite subgroup $H$ of $G$ has an abelian subgroup $H_1$
with the index $r := [H : H_1] \le \J$ (cf. \cite[Definition 1]{Po}).
Equivalently, we may even require $H_1$ to be normal in $H$, though we will {\it not} require so in this paper.
Indeed, consider the action of $H$ on $H/H_1$ by left multiplication.
This yields a homomorphism of $H$ to the symmetric group on $r$ letters.
The kernel of this homomorphism is a normal subgroup $H_2$ of $H$, of index dividing $r!$, and contained in $H_1$.
Define $\J(G)$ to be the
{\it smallest Jordan constant} for $G$; hence $G$ is Jordan if and only if $\J(G) < \infty$.
A family $\mathcal{G}$ of groups is
{\it uniformly Jordan} if there is a constant, denoted as $\J(\mathcal{G})$,
serving as a Jordan constant for every group in the family $\mathcal{G}$.
The question below was asked by Professor V.~L.~Popov.

\begin{question}\label{JQ} (cf.~\cite[\S 2, Problem A]{Po})
Let $X$ be a projective variety of dimension $n$.
Is the full automorphism group $\Aut(X)$ of $X$ a Jordan group?
\end{question}

For projective surfaces $X$, Question \ref{JQ} has been affirmatively answered by
V.~L.~Popov (except when $X$ is birational to $E \times \PP^1$ with $E$ an elliptic curve),
and Y.~G.~Zarhin (the remaining case); see \cite{Po}, \cite{Za14}, and the references therein. In higher dimensions, Y.~Prokhorov and C.~Shramov \cite[Theorem 1.8]{PS} proved the Jordan property for the birational automorphism group $\Bir(X)$, assuming either $X$ is non-uniruled,
or $X$ has vanishing irregularity and the Borisov-Alexeev-Borisov conjecture about the boundedness of terminal Fano varieties
which has just been affirmatively confirmed by C.~Birkar; see \cite[Theorem 1.1]{Bi}.

Our approach towards Jordan property for the full automorphism group $\Aut(X)$ of a projective variety $X$
in arbitrary dimension
is more algebraic-group theoretical.
It does not use the classification of projective varieties.

If $X$ is a projective variety, a classical result of
Grothendieck says that the neutral component $\Aut_0(X)$ of
$\Aut(X)$
is an algebraic group. Conversely, by M.~Brion \cite[Theorem 1]{Br13},
every connected algebraic group of dimension $n$ is isomorphic to $\Aut_0(X)$ for some
smooth projective variety $X$ of dimension $2n$.

In general, let $H$ be a connected algebraic group which is not necessarily linear. By
the classical result of Chevalley, there is a (unique) maximal connected linear algebraic normal subgroup
$L(H)$ of $H$ and an abelian variety $A(H)$ fitting the following exact sequence
$$1 \to L(H) \to H \to A(H) \to 1.$$
By the classical result of Camille Jordan, $L(H)$ is a Jordan group. Of course, $A(H)$ is also a Jordan group.
However, the extension of two Jordan groups may not be a Jordan group (cf. \cite[\S 1.3.2]{Po}).
Nevertheless, we would like to ask:

\begin{question}\label{qHom}
Let $H$ be an algebraic group. Is $H$ a Jordan group?
\end{question}

Clearly, it suffices to consider connected algebraic groups.
By \cite{Br13}, a positive answer to Question \ref{JQ} implies a positive answer to Question \ref{qHom}.
However, not every connected algebraic group is isogenous to
the product of an abelian variety and a connected linear algebraic group, diminishing the hope
to give a positive answer to Question \ref{qHom} by using the fact that the latter two groups
are Jordan groups (cf. Lemmas \ref{pres1} and \ref{pres2}, \cite[Remark 2.5 (ii)]{Br09}).

Conversely, a positive answer to Question \ref{qHom} implies a positive answer to Question \ref{JQ} by virtue of Lemma \ref{aa0}.

As in \cite[Lemma 2.8]{PS}, the Jordan property of a group $G$ is related to the {\it bounded rank property} of finite abelian subgroups of $G$: there exists a constant $C$ such that every finite abelian subgroup $H$ of $G$ is generated by $C$ elements. Denote by $\Rk(G)$ the {\it smallest constant} of such $C$, see Definition \ref{def} for details. Similarly we may define the {\it uniformly bounded rank property} of finite abelian subgroups for a family of groups.

Now we state our main results which positively answer the above questions.

\begin{theorem}\label{thalg}
Fix an integer $n \ge 0$.
Let $\mathcal{G}$ be the family of all connected algebraic groups of dimension $n$. Then $\mathcal{G}$ is uniformly Jordan;
its finite abelian subgroups have uniformly bounded rank. (See Theorem $\ref{thalg2}$ for upper bounds of $\J(\SG)$ and $\Rk(\SG)$ as functions in $n$.)
\end{theorem}

\begin{theorem}\label{thvar}
Fix an integer $n \ge 0$.
Let $\mathcal{G}=\{\Aut_0(X) \, | \, X$ is a projective variety of dimension $n\}$.
Then $\mathcal{G}$ is uniformly Jordan; its finite abelian subgroups have uniformly bounded rank.
(See Theorem $\ref{thvar2}$ for upper bounds of $\J(\SG)$ and $\Rk(\SG)$ as functions in $n$.)
\end{theorem}

\begin{remark}\label{nodirect}
Theorem \ref{thvar} does not follow from Theorem  \ref{thalg} directly because
$\dim \Aut_0(X)$ cannot be bounded in terms of $\dim X$. For example, by \cite[Theorem 3]{Ma71},
if $\Fd$ is the Hirzebruch surface of degree $d \ge 1$
then $G_a^{d+1}$ is the unipotent radical of $\Aut_0(\Fd)$
with $G_a=(k,+)$ the additive group, and $\dim \Aut_0(\Fd) = d+5$ is not bounded.
In order to prove Theorem \ref{thvar}, we remove the influence of such unipotent radical
by key Lemma \ref{derbound}.

Let $G$ be a connected algebraic group.
Denote by $G_{\aff}$ the largest connected affine normal subgroup.
Note that $G_{\aff}=L(G)$ .
Every anti-affine subgroup of $G$ is connected and contained in the center of $G$; in particular, every such subgroup is normal in $G$.
Denote by $G_{\ant}$ the largest anti-affine subgroup.
The main ingredient of our proof is the very old and classical decomposition theorem $G = G_{\aff}\cdot G_{\ant}$
for a connected algebraic group $G$.
This is due to M.~Rosenlicht \cite[Corollary 5, p.~440]{Ro56},
see also M.~Brion \cite{Br09} (or \ref{setup2} below) for more modern elaborations.

We also use the effective (and optimal) upper bound in M.~Brion \cite[Proposition 3.2]{Br13oc}
for the dimension of the anti-affine part of $\Aut_0(X)$, see Lemma \ref{abdim}.
\end{remark}

Our last main result below is an immediate consequence of Theorem \ref{thalg} and Lemma \ref{aa0}.

\begin{theorem}\label{thaut}
Let $X$ be a projective variety. Then $\Aut(X)$ is a Jordan group.
\end{theorem}

We remark that if $A$ is a non-trivial abelian variety and $Y$ is a non-trivial rational variety,
then $\Bir(A \times Y)$ is not a Jordan group (cf.~\cite[Corollary 1.4]{Za10}).
Hence the Jordan constant in Theorem \ref{thaut}, in general,
depends on $X$ and is not a birational invariant of $X$.

We refer to \cite{IMI} for related results.
\par \vskip 1pc \noindent
{\bf Acknowledgement.}

The authors heartily thank our friends and colleagues, and especially the referees of this paper, for providing Remark \ref{Referee}, improving Lemma \ref{4n2}, clarifying Corollary \ref{corbir}, removing the inaccuracies and many suggestions to improve the paper. The second-named author is supported by an ARF of NUS.

\section{Preliminary results}

\begin{setup}
{\rm
In this paper, $e_G$ or just $e$ denotes the identity element for a group $G$.
We denote by $Z(G)$ the {\it centre} of $G$.
For a (not necessarily linear) algebraic group $G$, we use
$G_0$ to denote the {\it neutral component} of $G$.
For a projective variety $X$,
$\NS(X)$ denotes the N\'eron-Severi group, and
$\NS_{\mathbb{Q}}(X) := \NS(X) \otimes_{\Z} {\mathbb{Q}}$.
}
\end{setup}

\begin{definition}\label{def} Given a group $G$, we introduce the following constants:
$$\begin{array}{lll}
&\Bd(G)=\sup\{|F|:|F|<\infty, F\leq G\},\\
&\Rk(G)=\sup\{\Rk(A):|A|<\infty, A \text{ is abelian}, A\leq G\},\\
\end{array}$$
 where $\Rk(A)$ is the {\it minimal number} of generators of a finite abelian group $A$.
 We define $\Rk(\{e_G\})=0$. Similarly we may define these constants for a family of groups.
\end{definition}

The easy observations below are frequently used.

\begin{lemma}\label{pres1}
Consider the exact sequence of groups
$$1 \to G_1 \to G \to G_2 \to 1.$$

\begin{itemize}
\item[(1)] $\J(G)\leq \Bd(G_2)\cdot \J(G_1)$.
\item[(2)] If $G_1$ is finite, then $\J(G_2)\leq \J(G)$.
\item[(3)] If $\Bd(G_1)=1$, then $\J(G)\le \J(G_2)$.
\item[(4)] If $G\cong G_1\times G_2$, then $\J(G)\leq \J(G_1)\cdot \J(G_2)$.
\item[(5)] $\Rk(G)\leq \Rk(G_1)+ \Rk(G_2)$.
\item[(6)] $\J(G)\leq \J(G_2)\cdot \Bd(G_1)^{\Rk(G_2)\cdot \Bd(G_1)}$.
\end{itemize}
\end{lemma}
\begin{proof} (1)-(5) are clear. For (6), we refer to the proof of \cite[Lemma 2.8]{PS}.
\end{proof}

\begin{lemma}\label{pres2} Below are some important constants.
\item[(1)](Jordan) Every general linear group $\GL_n(k)$ is a Jordan group. Hence every linear algebraic group is a Jordan group. Denote $\Con_n := \J(\GL_n(k))$. It is known that
    $$\Con_n < ((8n)^{1/2} + 1)^{2n^2}$$ (cf.~\cite[Theorem 36.14, p. 258-262]{CR}).
\item[(2)] (Minkowski) $\Bd(\GL_n(K))<\infty$ for any field $K$ which is finitely generated over $\mathbb{Q}$. The bound depends on $n$ and $[K:\mathbb{Q}]$ (cf. \cite[Theorem 5, and \S4.3]{Se07}).
\item[(3)] $\Rk(T) = \dim T$, when $T$ is an algebraic torus.
\item[(4)] $\Rk(\GL_n(k))=n$ (see \cite[\S 15.4 Proposition]{Hu} and use (3)).
\item[(5)] $\Rk(A)=2\dim A$, when $A$ is an abelian variety.
\end{lemma}

For any projective variety $X$ and its normalization $X'$, $\Aut(X)$ lifts to $X'$. So we may assume $X$ is normal in Theorem \ref{thaut}. Then, the following lemma and Theorem \ref{thvar} will imply Theorem \ref{thaut}.

\begin{lemma}\label{aa0}
Let $X$ be a normal projective variety.
Then there exists a constant $\ell$ (depending on $X$),
such that for any finite subgroup $G\leq \Aut(X)$, we have $[G:G\cap \Aut_0(X)] \le \ell$.
\end{lemma}

\begin{proof}
Take an $\Aut(X)$- (and hence $G$-) equivariant projective resolution $\pi: X' \to X$.
The action of $\Aut(X')$ on $X'$ induces a natural representation on $\NS_{\mathbb{Q}}(X')$.
Consider the exact sequence $$1\rightarrow K \rightarrow G\rightarrow G|_{\NS_{\mathbb{Q}}(X')}\rightarrow 1,$$
where $K$ is the kernel of the representation.
Note that $G|_{\NS_{\mathbb{Q}}(X')}\leq \GL_m(\mathbb{Q})$,
where $m = \dim_{\Q} \NS_{\mathbb{Q}}(X')$ is the Picard number of $X'$.
So by Lemma \ref{pres2}(2), there is a constant $\ell_1$, such that $|G|_{\NS_{\mathbb{Q}}(X')}| \le \ell_1$.
Thus we can find a subgroup $H \le G$ such that $[G:H] \le \ell_1$ and $H$ acts trivially on $\NS_{\mathbb{Q}}(X')$. In particular, $H$ preserves an ample divisor class $[H]$ of $X'$ and hence $H\le \Aut_{[H]}(X')$. By \cite[Proposition 2.2]{Li}, $\Aut_{[H]}(X')/\Aut_0(X')$ is finite. So there is a constant $\ell_2$, such that $[H:H \cap \Aut_0(X')] \le \ell_2$.

Now by \cite[Proposition 2.1]{Br11}, $\pi(\Aut_0(X')) \le \Aut_0(X)$ and we can identify $\Aut_0(X')$ with its $\pi$-image. Thus $H \cap \Aut_0(X') \le H \cap \Aut_0(X)$.
So $[H:H \cap \Aut_0(X)] \le [H:H \cap \Aut_0(X')] \le \ell_2$.
The lemma follows by setting $\ell = \ell_1 \cdot \ell_2$.
\end{proof}

\begin{remark}\label{Referee} Alternatively, one may show directly the following statement without taking resolution: let $X$ be
a projective variety, $L$ an ample line bundle on $X$, and $\Aut_{[L]}(X)$ the
subgroup of $\Aut(X)$ that fixes the class of $L$ in $\NS_{\mathbb{Q}}(X)$; then $\Aut_{[L]}(X)$ has
finitely many components. To see this, we identify $\Aut(X)$ with an open
subscheme of the Hilbert scheme $\Hilb(X \times X)$, by associating with each
automorphism $f$ its graph $\Gamma_f$ (cf.~\cite[Theorem 5.23]{FGI}). Also, $p_1^*(L)\otimes p_2^*(L)$ is an ample line bundle on
$X\times X$, and its pull-back to $\Gamma_f$ (identified with $X$) is $L\otimes f^*(L)$, which is numerically equivalent
to $L^2$ when $f \in \Aut_{[L]}(X)$. Replacing $L$ with a positive multiple, we may thus
assume that $L\otimes f^*(L)$ is algebraically equivalent to $L^2$. Then the Hilbert
polynomial of $\Gamma_f$ is independent of $f \in \Aut_{[L]}(X)$. Let $P$ be this polynomial.
Then $\Aut_{[L]}(X)$ is contained in $\Hilb_P (X\times X)$. Note that $\Hilb_P (X\times X)$ is a projective scheme by \cite[Chapter I, Theorem 1.4]{Ko} and $\Aut_{[L]}(X)\cap C$ is closed in $C$ for any connected component $C$ of $\Aut(X)$. Hence, $\Aut_{[L]}(X)$ is a quasi-projective scheme which yields the assertion.
So Lemma \ref{aa0} works for non-normal $X$.
\end{remark}
We need a few more results from algebraic group theory. The proofs are easy but we give proofs here for the convenience of the reader.

%

\begin{lemma}\label{zssmaller}
Let $p: \tilde{S}\rightarrow S$ be an isogeny between two semisimple linear algebraic groups. Then $|Z(S)|\leq |Z(\tilde{S})|$.
\end{lemma}

\begin{lemma}\label{simplebound}
Let $G$ be a connected almost simple linear algebraic group.
Then we have:
$$|Z(G)| \le \Rank G+1<\dim G < 4(\Rank G)^2$$
where the rank $\Rank G$ of $G$
equals $\dim T$ of a (and every) maximal algebraic torus $T$ contained in $G$.
\end{lemma}

\begin{proof}
There is an isogeny $p:\tilde{G}\rightarrow G$ such that $\tilde{G}$ is simply connected and almost simple.
By Lemma \ref{zssmaller} and replacing $G$ by $\tilde{G}$,
we may assume further that $G$ is simply connected.

Up to isomorphism, there is a 1-1 correspondence between simply connected almost simple linear algebraic groups
$G$ and the Dynkin diagrams given in the following table, which also shows their centres and ranks.
The lemma follows from the table.

 \begin{center}
\begin{tabular}{|c|c|c|c|c|c|c|c|c|c|}
\hline
$D(G)$&$Z(G)$&$\dim G$&$\Rank G$\\
\hline
$A_{\ell}, {\ell}>0$&$\mathbb{Z}_{{\ell}+1}$&${\ell}({\ell}+2)$&${\ell}$\\
\hline
$B_{\ell}, {\ell}>1$&$\mathbb{Z}_2$&${\ell}(2{\ell}+1)$&${\ell}$\\
\hline
$C_{\ell}, {\ell}>2$&$\mathbb{Z}_2$&${\ell}(2{\ell}+1)$&${\ell}$\\
\hline
$D_{2{\ell}}, {\ell}>1$&$\mathbb{Z}_{2}\bigoplus\mathbb{Z}_{2}$&$2{\ell}(4{\ell}-1)$&$2{\ell}$\\
\hline
$D_{2{\ell}+1},{\ell}>1$&$\mathbb{Z}_4$&$(2{\ell}+1)(4{\ell}+1)$&$2{\ell}+1$\\
\hline
$E_6$&$\mathbb{Z}_3$&$78$&$6$\\
\hline
$E_7$&$\mathbb{Z}_2$&$133$&$7$\\
\hline
$E_8,F_4,G_2$&$\{e_G\}$&$248,52,14$&$8,4,2$\\
\hline
\end{tabular}
 \end{center}
\end{proof}

\begin{lemma}\label{semibound}
Let $S$ be a connected semisimple linear algebraic group of dimension $n$. Then $|Z(S)| \leq n^{n}$.
\end{lemma}
\begin{proof} Let $\{S_i\}_{i=1}^m$ be the minimal closed connected normal subgroups of positive dimension.
(We set $m := 0$ when $S$ is trivial.)
 The natural product map gives an isogeny $\prod_{i=1}^m S_i\rightarrow S$,
 with kernel contained in the centre of the domain of the map.
 By Lemmas \ref{zssmaller} and \ref{simplebound}, we have
 $$|Z(S)|\leq \prod_{i=1}^m |Z(S_i)| \leq n^m \le n^{n} .$$
\end{proof}

\begin{remark}\label{rmkbound}
It is well known that, up to isomorphism, there are only finitely many $n$-dimensional semisimple linear algebraic groups.
Thus there is a function $\N(n)$, such that every connected semisimple linear algebraic group of dimension $\le n$
can be embedded into $\GL_{\N(n)}(k)$. Denote by $\Se(n)$ the {\it supremum} of Jordan constants
for all connected semisimple linear algebraic groups of dimension $\le n$. Clearly $\Se(n)\leq \Con_{\N(n)}$.
\end{remark}

\section{Proof of Theorems }

In this section, we will prove Theorems \ref{thalg2} and \ref{thvar2}
which are the precise versions of Theorems \ref{thalg} and \ref{thvar}.
The Jordan constant and the uniformly bounded rank in these theorems could be made more optimal
at the expense of more complicated expressions, but they have not been done by us.

\begin{setup}\label{setup2}
{\rm
Here we give some notations and facts first. \\
\begin{itemize}
\item[(1)]
Let $G$ be a group. $G^{(1)}=(G,G)$ denotes the {\it commutator subgroup} of $G$. \\
\item[(2)]
Let $G$ be a connected algebraic group. We use the conventions and facts as in
\cite[\S 1]{Br09} (see also \cite{Ro56}).
$G_{\aff}$ and $G_{\ant}$ denote respectively
the {\it affine part} and {\it anti-affine part} of $G$,
both being connected and normal in $G$. We have the
{\it Rosenlicht decomposition}:
$$G = G_{\aff}\cdot G_{\ant}, \hskip 1pc G_{\ant}\leq Z(G)_0 .$$
Further, $G/G_{\aff} \cong G_{\ant}/(G_{\ant} \cap G_{\aff})$ is an abelian variety (the albanese variety of $G$), and
$G/G_{\ant} \cong G_{\aff}/(G_{\aff} \cap G_{\ant})$ is the largest affine quotient group of $G$.
\item[(3)]
Let $G$ be a connected algebraic group. Denote by $R_u(G)$ the {\it unipotent radical} of $G_{\aff}$ and $G_r$
a {\it Levi subgroup} of $G_{\aff}$ so that we have (cf.~\cite{Mo56}):
$$G_{\aff}=R_u(G)\rtimes G_r .$$
Levi subgroups of $G_{\aff}$ are all $R_u(G)$-conjugate to each other, and $G_r$ is one
of them which we fix.
\item[(4)]
Let $G$ be a connected reductive linear algebraic group.
Then
$$G=R(G)\cdot G^{(1)}$$
where $R(G)$ is the solvable radical of $G$ (an algebraic torus now),
$R(G)=Z(G)_0$, and $G^{(1)}$ is semisimple and connected,
see \cite[\S 19]{Hu}.
\item[(5)]
Let $G$ be a connected linear algebraic group and $N \le G$ a closed normal subgroup
with $\gamma: G \to G/N$ the quotient map.
Then the $\gamma$-image of a Levi subgroup of $G$ is a Levi subgroup of $G/N$.
\item[(6)]
Every nontrivial unipotent element of a linear algebraic group has infinite order,
because our ground field $k$ has characteristic zero.
\end{itemize}
}
\end{setup}
\begin{lemma}\label{4n2} Let $G$ be a connected reductive linear algebraic group. Then $\dim G\le 4(\Rank G)^2$.
\end{lemma}
\begin{proof} The product map gives an isogeny $Z(G)_0 \times G^{(1)}\rightarrow G$. Since $G^{(1)}$ is semisimple, there is an isogeny $\prod_{i=1}^m S_i\rightarrow G^{(1)}$ with $S_i$ connected and almost simple. Hence we have an isogeny $$\pi:Z(G)_0 \times \prod_{i=1}^m S_i\rightarrow G.$$
Take a maximal torus $T_i$ of $S_i$ and note that $Z(G)_0$ is an algebraic torus. Then $\pi(Z(G)_0 \times\prod_{i=1}^m T_i)$
is an algebraic torus of $G$. So $\Rank Z(G)_0+\sum_{i=1}^m \Rank S_i\le \Rank G$. By Lemma \ref{simplebound}, $\dim S_i<4(\Rank S_i)^2$. So
$$\begin{array}{lll}
  \dim G&=\dim Z(G)_0+\sum_{i=1}^m \dim S_i\\
  &\le  \Rank Z(G)_0+\sum_{i=1}^m 4(\Rank S_i)^2\\
  &\le 4(\Rank Z(G)_0+\sum_{i=1}^m \Rank S_i)^2\\
  &\le 4(\Rank G)^2.
  \end{array}$$
\end{proof}

\begin{lemma}\label{redrank}
Let $G$ be a connected reductive linear algebraic group with $\dim G \leq n$.
Then $\Rk(G) \leq n+\N(n)$.
\end{lemma}

\begin{proof} We refer to Lemma \ref{pres1} (5) to give a proof by reduction.
Consider the exact sequence $$1\rightarrow Z(G)_0\rightarrow G\rightarrow G/Z(G)_0\rightarrow 1 .$$
Then
$$(*) \hskip 1pc \Rk(G)\leq \Rk(Z(G)_0)+\Rk(G/Z(G)_0) .$$
Since $Z(G)_0$ is an algebraic subtorus of $G$, Lemma \ref{pres2} (3) implies
$\Rk(Z(G)_0)\leq \dim Z(G)_0 \le n$.
Note that
$$G/Z(G)_0\cong G^{(1)}/(Z(G)_0\cap G^{(1)})$$
is semisimple connected and of dimension $\le \dim G \le n$.
By Remark \ref{rmkbound}, $G/Z(G)_0$ can be embedded into $\GL_{\N(n)}(k)$. So by Lemma \ref{pres2} (4),
$\Rk(G/Z(G)_0)\leq \N(n)$. Combining the above two inequalities about ranks,
we get the lemma, via the above display $(*)$.
\end{proof}

\begin{lemma}\label{redbound}
Let $G$ be a connected reductive linear algebraic group with $\dim G^{(1)}\leq n$. Then $\J(G)\leq \Se(n)$.
\end{lemma}
\begin{proof}
The product map gives an isogeny $ Z(G)_0 \times G^{(1)}\rightarrow G$.
Thus Lemma \ref{pres1} (2) and the commutativity of $Z(G)_0$ imply (cf.~Remark \ref{rmkbound}):
$$\J(G)\leq \J(Z(G)_0\times G^{(1)}) = \J(G^{(1)}) \le \Se(n) .$$
\end{proof}

\begin{corollary}\label{linearbound}
Let $G$ be a connected linear algebraic group with $\dim (G_r)^{(1)}\leq n$. Then $\J(G/N) \le \Se(n)$ for any closed normal subgroup $N$ of $G$.
\end{corollary}

\begin{proof} First consider the exact sequence $$1\rightarrow R_u(G)\rightarrow G\rightarrow G_r\rightarrow 1,$$ where $R_u(G)$ is the unipotent radical of $G$. Since $\Bd(R_u(G))=1$, by Lemmas \ref{pres1} (3) and \ref{redbound}, $\J(G) \le \J(G_r)\leq \Se(n)$.
Note that $\dim ((G/N)_r)^{(1)} \le \dim (G_r)^{(1)} \le n$, see \ref{setup2}. Hence $\J(G/N) \le \Se(n)$ also holds.
\end{proof}

\begin{remark}
Lemma \ref{redbound} slightly extends \cite[Theorem 15]{Po}.
\end{remark}

Below is our key lemma. The proof crucially utilizes the Rosenlicht decomposition $G = G_{\aff} \cdot G_{\ant}$
as in \ref{setup2}. Recall $G_{\aff}=R_u(G)\rtimes G_r$.

\begin{lemma}\label{derbound} Let $G$ be a connected algebraic group with $\dim (G_r)^{(1)}\leq n$ and
let $H\leq G$ be a finite subgroup. Then there exists a subgroup $H_1\leq H$ such that the index $[H:H_1]\leq \Se(n)$, $H_1^{(1)}\leq Z(G)$ and $|H_1^{(1)}|\leq n^n$.
\end{lemma}

\begin{proof}
Consider the natural homomorphism
$$\varphi:G\rightarrow G/G_{\aff}\times G/G_{\ant}$$
with
$$\Ker \varphi=G_{\aff}\cap G_{\ant}\leq G_{\ant} \leq Z(G) .$$
Since $G/G_{\aff}$ is abelian and by Corollary \ref{linearbound}, we have:
$$\J(G/G_{\aff}\times G/G_{\ant})=\J(G/G_{\ant})
= \J(G_{\aff}/(G_{\aff} \cap G_{\ant})) \leq \Se(n) .$$
So there exists a subgroup $H_1\leq H$ such that the index $[H:H_1]\leq \Se(n)$ and $\varphi(H_1)$ is abelian.

Note that $G_r=Z(G_r)\cdot G_s$, where $G_s := (G_r)^{(1)}$ is connected semisimple. Thus we have a natural injective homomorphism $i: H_1^{(1)}\rightarrow G_s$ from the following commutative diagram.
$$\xymatrix{
  H_1^{(1)} \ar@{^{(}->}[r]\ar@{>->}[rrd]_{i} &G^{(1)}=(G_{\aff})^{(1)} \ar@{->>}[r] \ar@{->>}[rd]^{p}
  & (G_{\aff})^{(1)}/(R_u(G) \cap (G_{\aff})^{(1)}) \ar[d]^{\cong}\\
   & & (G_r)^{(1)}=G_s
  }
$$
where $G^{(1)}= (G_{\aff})^{(1)}$ because $G=G_{\aff} \cdot G_{\ant}$ and $G_{\ant}\le Z(G)$. The surjective group homomorphism $p$ is the restriction of the group homomorphism $G_{\aff}\twoheadrightarrow G_r$ to $(G_{\aff})^{(1)}$, and $\Ker p=R_u(G) \cap (G_{\aff})^{(1)}$. Note that $\Ker i=\Ker p\cap H_1^{(1)}\leq R_u(G)\cap H_1^{(1)}=\{e_G\}$ since $H_1^{(1)}$ is finite (cf.~\ref{setup2} (6)). So $i$ is injective.

Since $\varphi(H_1)$ is abelian,
we have $H_1^{(1)}\leq \Ker \varphi\leq Z(G)$.
This and the above diagram imply $i(H_1^{(1)})\leq Z(G_s)$.
Thus $|H_1^{(1)}|\leq |Z(G_s)| \leq n^n$, since $G_s = (G_r)^{(1)}$
has dimension $\le n$, and by Lemma \ref{semibound}.
\end{proof}

\begin{lemma}\label{rankbound} Let $G$ be a connected algebraic group with $\dim G_r\leq n$ and $\dim G_{\ant}\leq m$. Then for any finite subgroup $H$ with $H^{(1)}\leq Z(G)$, we have $\Rk(G/H^{(1)})\leq 2m+n+\N(n)$.
\end{lemma}

\begin{proof} We refer to Lemma \ref{pres1} (5) to give a proof by reduction.
First, as in the diagram in Lemma \ref{derbound},
we have $H^{(1)}\le (G_{\aff})^{(1)} \le G_{\aff} = R_u(G)\rtimes G_r$. We claim that $H^{(1)}\le G_r$.
Note that for any $x,y \in H$, we can write $xyx^{-1}y^{-1}=ur$, for some $u\in R_u(G)$ and $r\in G_r$.
Since $ur \in Z(G)$, we have $(ur)^2 = u(ur)r = u^2r^2$, and inductively $(ur)^t = u^tr^t$ which equals $e_G$ for some $t\in \mathbb{Z}_{>0}$. 
Hence $u^t = r^{-t} \in G_r \cap R_u(G) = \{e_G\}$,
and then $u=e_G$ (cf.~\ref{setup2} (6)).
Thus $xyx^{-1}y^{-1}\in G_r$.
This proves the claim.
Clearly $H^{(1)}$ is normal in $G$, since $H^{(1)}\leq Z(G)$.

Now consider the exact sequence $$1\rightarrow G_{\aff}/H^{(1)}\rightarrow G/H^{(1)}\rightarrow G/G_{\aff}\rightarrow 1,$$ where $G/G_{\aff}$ is an abelian variety.
By Lemma \ref{pres1}(5), we have
$$(*) \hskip 1pc
\Rk(G/H^{(1)})\leq \Rk(G_{\aff}/H^{(1)})+\Rk(G/G_{\aff}) .$$
Note that
$G/G_{\aff} \cong G_{\ant}/(G_{\ant} \cap G_{\aff})$ (cf.~\ref{setup2}).
So we have $\Rk(G/G_{\aff})=2\dim G/G_{\aff}\leq 2\dim G_{\ant} \le 2m$ by Lemma \ref{pres2} (5).
For $G_{\aff}/H^{(1)}$, the Levi decomposition $G_{\aff}\cong R_u(G_{\aff})\rtimes G_r$
and $H^{(1)}\leq G_r \cap Z(G)$ imply
$$G_{\aff}/H^{(1)}\cong R_u(G)\rtimes G_r/H^{(1)} .$$
Thus
$$\Rk(G_{\aff}/H^{(1)})\leq \Rk(R_u(G))+\Rk(G_r/H^{(1)}) .$$
Since $R_u(G)$ is a unipotent group, we have $\Rk(R_u(G))=0$ (cf.~\ref{setup2} (6)).
For $G_r/H^{(1)}$, this is also a connected reductive group with $\dim G_r/H^{(1)} = \dim G_r\leq n$.
By Lemma \ref{redrank}, $\Rk(G_r/H^{(1)})\leq n+\N(n)$.
Now the lemma follows from the above display $(*)$ and the two inequalities about ranks we just obtained.
\end{proof}

The theorem below is the precise version of Theorem \ref{thalg}.

\begin{theorem}\label{thalg2}
Let $G$ be a connected algebraic group of dimension $n$.
Then we have:
$$\J(G)\leq \Se(n)\cdot (n^n)^{(3n+\N(n))\cdot n^n}, \hskip 1pc \Rk(G)\leq 3n+\N(n),$$
where $\Se(n), \N(n)$ are defined in Remark \ref{rmkbound}.
\end{theorem}

\begin{proof}
$\Rk(G)\leq 3n+\N(n)$ is straightforward by Lemma \ref{rankbound}.

Let $H\leq G$ be a finite subgroup.
By Lemma \ref{derbound}, there exists a subgroup $H_1\leq H$ with $[H:H_1]\leq \Se(n)$ such that  $H_1^{(1)}\leq Z(G)$ and $|H_1^{(1)}|\leq n^n$. Note that $H_1^{(1)}$ is normal closed in $G$. So $H_1/H_1^{(1)}\leq G/H_1^{(1)}$ is a finite abelian subgroup of a connected algebraic group of dimension $n$. By Lemma \ref{rankbound}, $\Rk(H_1/H_1^{(1)})\leq 3n+\N(n)$.
Applying Lemma \ref{pres1} (6) to
$$1 \to H_1^{(1)} \to H_1 \to H_1/H_1^{(1)} \to 1 ,$$
we can find an abelian subgroup $A\leq H_1$ with $[H_1:A]\leq (n^n)^{(3n+\N(n))\cdot n^n}$.
Clearly, $[H:A]\leq \Se(n)\cdot (n^n)^{(3n+\N(n))\cdot n^n}$.
\end{proof}

As we discussed in Remark \ref{nodirect}, before giving the proof for Theorem \ref{thvar}, we need the following result
which is proved in \cite[\S 1, Proposition 7(b)]{De70}.

\begin{lemma}\label{torusdim} Let $T$ be an algebraic torus acting faithfully on a projective variety $X$.
Then $T$ acts generically freely on $X$: the stabilizer subgroup $T_x$ is trivial for general point $x \in X$.
In particular, $\dim T = \dim T x \leq \dim X$.
\end{lemma}

\begin{lemma}\label{reddim}
Let $G$ be a connected reductive linear algebraic group acting faithfully on a projective variety $X$ with $\dim X=n$.
Then $\dim G\leq 4n^2$.
\end{lemma}

\begin{proof} Let $T$ be a maximal torus of $G$. Then $\dim G\le 4(\dim T)^2$ by Lemma \ref{4n2}. By Lemma \ref{torusdim}, $\dim T\le n$. So $\dim G\le 4n^2$.
\end{proof}

We also need the following effective (and optimal) bound of \cite[Proposition 3.2]{Br13oc}.

\begin{lemma}\label{abdim}
Let $G$ be an anti-affine algebraic group acting faithfully on a projective variety $X$. Then $\dim G\leq 2\dim X$.
\end{lemma}

The theorem below is the precise version of Theorem \ref{thvar}.

\begin{theorem}\label{thvar2} Let $X$ be a projective variety of dimension $n$.
Then we have:
$$\J(\Aut_0(X))\leq \Se(t) \cdot (t^t)^{(4n+t+\N(t))\cdot t^t}, \hskip 1pc
\Rk(\Aut_0(X))\leq 4n+t+\N(t)$$
where $t=4n^2$ and $\Se(t), \N(t)$ are defined in Remark \ref{rmkbound}.
\end{theorem}

\begin{proof}
Let $G=\Aut_0(X)$ and $G_r$ a Levi subgroup of $G_{\aff}$.
By Lemmas \ref{reddim} and \ref{abdim}, $\dim G_r\leq t := 4n^2$ and $\dim G_{\ant}\leq 2n$.
Then we are done by the same argument in the proof of Theorem \ref{thalg2}.
\end{proof}

\begin{corollary}\label{corbir}
Let $X$ be a projective variety of dimension $n$ and $G$ a connected algebraic group contained in $\Bir(X)$.
Setting $t=4n^2$, we have:
$$\J(G)\leq \Se(t) \cdot (t^t)^{(4n+t+\N(t))\cdot t^t}, \hskip 1pc \Rk(G)\leq 4n+t+\N(t),$$
where $\Se(t), \N(t)$ are defined in Remark \ref{rmkbound}.
\end{corollary}
\begin{proof} By \cite[Theorem 1]{Ro56}, there exists a variety $X_0$ birational to $X$ such that $G$ acts on $X_0$ biregularly. Since the smooth locus of $X_0$ is $G$-stable, we may assume that $X_0$ is smooth.
Then $X_0$ is covered by $G$-stable quasi-projective open subsets by \cite[Theorem 1]{Br07}.
So we may further assume that $X_0$ is quasi-projective.
Then $X_0$ admits a $G$-equivariant embedding into the projectivization of a $G$-homogeneous vector bundle over an abelian variety by \cite[Theorem 2]{Br07}. In particular, $X_0$ admits an $G$-equivariant embedding into a projective $G$-variety $Y$. Let $X'$ be the closure of $X_0$ in $Y$. Then $X'$ is also a projective $G$-variety which is birational to $X$. So $G \le \Aut_0(X')$ and $\dim X'=n$.
The result then follows from Theorem \ref{thvar2}.
\end{proof}


\begin{thebibliography}{99}

\bibitem{Bi} C.~Birkar,
Singularities of linear systems and boundedness of Fano varieties,
\href{http://arXiv.org/abs/1609.05543v1}{\tt arxiv:1609.05543v1}.

\bibitem{Br07} M.~Brion,
Some basic results on actions of nonaffine algebraic groups, in: Symmetry
and spaces, 1-20, Progr. Math. \textbf{278}, Birkh\"auser, Boston, MA, 2010.

\bibitem{Br09} M.~Brion,
On the geometry of algebraic groups and homogeneous spaces,
J. Algebra \textbf{329} (2011), 52-71.

\bibitem{Br11} M.~Brion,
On automorphism groups of fiber bundles,
Publ. Mat. Urug. \textbf{12} (2011), 39-66.

\bibitem{Br13oc} M.~Brion,
On connected automorphism groups of algebraic varieties,
J. Ramanujan Math. Soc. \textbf{28A} (2013), 41-54.

\bibitem{Br13} M.~Brion,
On automorphisms and endomorphisms of projective varieties, in: Automorphisms in birational and affine geometry, 59-81,
Springer Proc. Math. Stat., \textbf{79}, Springer, Cham, 2014.

\bibitem{CR} C.~W.~Curtis and I.~Reiner,
Representation theory of finite groups and associative algebras,
Pure and Applied Mathematics, Vol. \textbf{XI}. Interscience Publishers, a division of John Wiley \& Sons, New
York-London, 1962.

\bibitem{De70} M. Demazure, Sous-groupes alg\'ebriques de rang maximum du groupe de Cremona, (French)
Ann. Sci. \'Ecole Norm. Sup. (4) \textbf{3} (1970), 507-588.

\bibitem{FGI}
B.~Fantechi, L.~G\"ottsche, L.~Illusie, S.~L.~Kleiman, N.~Nitsure and A.~Vistoli,
Fundamental algebraic geometry,
Grothendieck's FGA explained.
Mathematical Surveys and Monographs, \textbf{123}. American Mathematical Society, Providence, RI, 2005. x+339 pp.

\bibitem{Hu}
J.~ E.~ Humphreys,
Linear Algebraic Groups,
Graduate Texts in Mathematics, Volume \textbf{21},
Springer, 1975.

\bibitem{IMI}
I.~Mundet~i~Riera,
Finite subgroups of Ham and Symp,
\href{http://arXiv.org/abs/1605.05494}{\tt arxiv:1605.05494}.
%
%
%
\bibitem{Ko}J. Koll\'ar,
Rational curves on algebraic varieties,
Ergebnisse der Mathematik und ihrer Grenzgebiete. 3. Folge. A Series of Modern Surveys in Mathematics \textbf{32}. Springer-Verlag,
Berlin, 1996.
\bibitem{Li}
D.~I.~Lieberman,
Compactness of the Chow scheme: applications to automorphisms
and deformations of K\"ahler manifolds,
\emph{Fonctions de plusieurs variables complexes, III}
(\emph{S\'em.\ Fran\c{c}ois Norguet, 1975--1977}), pp.~140--186,
Lecture Notes in Math., \textbf{670}, Springer, Berlin, 1978.

\bibitem{Ma71}
M.~Maruyama,
On automorphism groups of ruled surfaces,
J. Math. Kyoto Univ.
\textbf{11}, No. 1 (1971), 89-112.

\bibitem{Mo56}
G.~Mostow,
Fully reducible subgroups of algebraic groups,
Amer. J. Math. \textbf{78} (1956), 200-221.

\bibitem{Po}
V.~L.~Popov,
Jordan groups and automorphism groups of algebraic varieties, in: Automorphisms in birational and affine geometry, 185-213,
Springer Proc. Math. Stat., \textbf{79}, Springer, Cham, 2014.

\bibitem{PS}
Y.~Prokhorov and C.~Shramov,
Jordan property for groups of birational selfmaps,
Compos. Math. \textbf{150} (2014), no. 12, 2054-2072.

\bibitem{Ro56}
M.~Rosenlicht, Some basic theorems on algebraic groups, Amer.\ J. \ Math.
\textbf{78} (1956), 401--443.

\bibitem{Se07}
J.~P.~Serre, Bounds for the orders of the finite subgroups of $G(k)$, Group representation theory, 405-450, EPFL Press, Lausanne, 2007.

\bibitem{Za10}
Y.~G. Zarhin,
Theta groups and products of abelian and rational varieties,
Proc. \ Edinburgh Math. \ Soc. \ \textbf{57}, issue 1 (2014), 299-304.

\bibitem{Za14}
Y.~G.~Zarhin,
Jordan groups and elliptic ruled surfaces,
Transform. Groups, \textbf{20} (2015), no. 2, 557-572.


%
%
%
%

\end{thebibliography}
\end{document}